\documentclass[11pt]{amsart}
\usepackage{amssymb}

\newtheorem{theorem}{Theorem}

\newtheorem{prop}[theorem]{Proposition}
\newtheorem{corollary}[theorem]{Corollary}

\theoremstyle{definition}

\newtheorem{example}[theorem]{Example}

\theoremstyle{remark}

\newcommand{\DD}{{\mathbb D}}
\newcommand{\OO}{{\mathcal O}}

\newcommand{\LL}{{\mathcal L}}
\newcommand{\MM}{{\mathcal M}}
\newcommand{\NN}{{\mathbb N}}

\newcommand{\RR}{{\mathbb R}}
\newcommand{\CC}{{\mathbb C}}

\renewcommand{\phi}{\varphi}

\begin{document}

\title{Carath\'eodory completeness on the complex plane}

\author{Armen Edigarian}

\address{Jagiellonian University, Faculty of Mathematics and Computer Science,
\L ojasiewicza~6, 30-348 Krak\'ow, Poland}
\email{armen.edigarian@uj.edu.pl}
\thanks{The author was supported in part by the Polish National Science Centre (NCN) grant
no.~2015/17/B/ST1/00996}



\maketitle

\begin{center}
{\it Dedicated to the memory of Professor J\'ozef Siciak}
\end{center}

\begin{abstract} In 1975 N. Sibony \cite{Sib1} and, independently, M. A. Selby \cite{Sel1} proved that on the complex plane $c$-completeness is equivalent to $c$-finitely compactness. 
In the paper we give a local version of their results. We also simplify the proofs.
\end{abstract}

\maketitle

\section{Introduction}
Let $\DD=\{\lambda\in\CC: |\lambda|<1\}$ denote the unit disk. For $\lambda',\lambda''\in\DD$ we define the M\"obius function $m$ on $\DD$ as 	
$$
m(\lambda', \lambda'')=\left|\frac{\lambda'-\lambda''}{1-\bar \lambda'\lambda''}\right|,
$$
and the Poincar\'e function $p$ as $p(\lambda',\lambda'')=\frac12\log\frac{1+m(\lambda',\lambda'')}{1-m(\lambda',\lambda'')}$.

Let $X$ be a complex manifold. For $x,y\in X$ put
\begin{align*}
c_X(x,y)=&\sup\{p(f(x),f(y)): f\in\OO(X;\DD)\},\\
c_X^\ast(x,y)=&\sup\{m(f(x),f(y)): f\in\OO(X;\DD)\},
\end{align*}
where $\OO(X;\DD)$ denotes the set of all holomorphic mappings $X\to\DD$. The function $c_X$ is called the Carath\'eodory pseudodistance for $X$ (see e.g. \cite{JP1}). In case, when $c_X$ is indeed distance we say that $X$ is $c$-hyperbolic. A $c$-hyperbolic manifold $X$ is called $c$-complete if any $c_X$-Cauchy sequence $\{x_\nu\}_{\nu\ge1}\subset X$ converges to a point $x_0\in X$ (w.r.t. "usual" topology). We say that a $c$-hyperbolic manifold $X$ is $c$-finitely compact if any ball
$B_c(x_0,R)=\{x\in X: c_X(x,x_0)<R\}$ is relatively compact in $X$ (w.r.t. "usual" topology). It is known (see \cite{Sib1} and \cite{Sel1}) that on a domain in the complex plane these two notions are equivalent.
The aim of this paper is to give a local version of the results of N. Sibony \cite{Sib1} and M. A. Selby \cite{Sel1}, we also simplify the proofs.
\begin{theorem}\label{thm:1}
Let $\Omega\subset\CC$ be a domain and let $\zeta\in\partial\Omega$ be a boundary point. Then the following conditions are equivalent:
\begin{enumerate}
\item\label{thm:1:1} the point $\zeta$ is a weak peak function for $\Omega$, i.e., there exists an $f\in \OO(\Omega)\cap C(\Omega\cup\{\zeta\})$ such that $|f|<1$ on $\Omega$ and $f(\zeta)=1$;
\item\label{thm:1:4} there exist no Borel finite measure $\mu$ on $\Omega$ such that
$$
|f(\zeta)|\le\int|f|d\mu\quad\text{ for any } f\in \OO(\Omega)\cap C(\Omega\cup\{\zeta\}).
$$
\item\label{thm:1:5} there exist no Borel probability measure $\mu$ on $\Omega$ such that
$$
f(\zeta)=\int f d\mu\quad\text{ for any } f\in \OO(\Omega)\cap C(\Omega\cup\{\zeta\}).
$$
\item\label{thm:1:3} 
fix any $a\in(0,1)$. Then 
$$
\sum_{n=1}^\infty \frac{\gamma(A_n(\zeta,a)\setminus\Omega)}{a^n}=+\infty,
$$
where $A_n(\zeta,a)=\{z\in\CC: a^{n+1}\le |z-\zeta|\le a^n\}$ and $\gamma$ is the analytic capacity (see \cite{Gam} and definition below);
\item\label{cond:1} for any sequence $\{z_\nu\}_{\nu\ge1}\subset\Omega$ such that $z_\nu\to\zeta$ we have $c_{\Omega}(z_0,z_\nu)\to\infty$;
\item\label{cond:3} for any sequence $\{z_\nu\}_{\nu\ge1}\subset\Omega$ such that $z_\nu\to\zeta$ we have $\{z_\nu\}$ is not a $c_{\Omega}$-Cauchy sequence;

\item\label{cond:4} for any sequence $\{z_\nu\}_{\nu\ge1}\subset\Omega$ with $z_\nu\to\zeta$ there exists an $f\in\OO(\Omega)$ such that
$|f|<1$ on $\Omega$ and $f(z_\nu)\to1$ when $\nu\to\infty$.
\end{enumerate}
\end{theorem}
Part of Theorem~\ref{thm:1} for a domain $\Omega$ in $\CC^N$ is claimed in \cite{Edi}. However, the presented proof was based on a version of Hahn-Banach theorem given in \cite{Pol-Gog}. At the end of the paper we give a simple example, which shows that this version of Hahn-Banach does not hold. So, at the moment we do not know whether Theorem~\ref{thm:1} holds true in higher dimension.

\section{Conditions for existence of a peak function}
Let $F$ be any subset of $\CC$. 
We denote by $R(F)$ the set of all continuous functions on $F$, which can be approximated by holomorphic functions on a neighborhood of $F$. 
We say that $x_0\in F$ is a peak point for $R(F)$ if there exists a function $f\in R(F)$ such that $|f|<1$ on $F\setminus\{x_0\}$ and $f(x_0)=1$.
The next result  (see e.g., Theorem 10.8 in Chapter VIII \cite{Gam}) gives a description of $R(\Omega\cup\{\zeta\})$ for a domain $\Omega\subset\CC$ and a boundary point $\zeta\in\partial\Omega$. 
\begin{theorem} Let $\Omega\subset\CC$ be a domain and let $\zeta\in\partial\Omega$. 
For any $f\in H^\infty(\Omega)$
there exists a sequence $\{f_n\}_{n\ge1}\subset H^\infty(\Omega)$ with $\|f_n\|_{\Omega}\le 17\|f\|_{\Omega}$ such that $f_n\to f$ locally uniformly on $\Omega$ and each $f_n$ extends holomorphically to a neighborhood of $\zeta$. Moreover, if $f$ extends continuously to $\zeta$ then $f_n$ tends to $f$ uniformly on $\Omega$.

In particular,  $R(\Omega\cup\{\zeta\})=\OO(\Omega)\cap C(\Omega\cup\{\zeta\})$.
\end{theorem}

Let us give a variation of 1/4-3/4 Bishop's theorem (cf. Theorem 11.1 in Chapter II \cite{Gam}), which is a standard technique in a construction of peak functions.
\begin{theorem}\label{theorem:10a} Let $X$ be a topological space
and let $x_0\in X$. Assume that there exists a sequence $\{f_\nu\}_{\nu}\subset C(X)$ and numbers $0<r<1\le R$ such that
\begin{itemize}
\item $|f_\nu|\le R$ on $X$;
\item $f_\nu(x_0)=1$;
\item for any $x\in X\setminus\{x_0\}$ there exists a $k=k(x)$ such that $|f_\nu(x)|<1$ for any $\nu\ge k$.
\item for any $k\in\NN$ and any $\epsilon>0$ there exists an $m_0$ such that
for any $m\ge m_0$ we have
$|f_m|\le r$ on the set $\{z\in\Omega:|f_k(z)|\ge 1+\epsilon\}$.
\end{itemize}
Then for any $s\in (\frac{R-1}{R-r},1)$ there exists a subsequence $\{n_k\}_{k=1}^\infty$ such that
for a function $h(x)=(1-s)\sum_{k=1}^\infty s^k f_{n_k}(x)$ we have
\begin{enumerate}
\item $|h|<1$ on $X\setminus\{x_0\}$;
\item $h(x_0)=1$;
\item $F_N$ tends uniformly to $h$ on $X$, where $F_N=(1-s)\sum_{k=1}^N s^k f_{n_k}$.
\end{enumerate}
\end{theorem}

\begin{proof} Fix $s\in (\frac{R-1}{R-r},1)$.
Then $R-1+s(r-R)<0$. Choose a sequence
$\epsilon_\nu\searrow0$ such that
$$
\epsilon_{\nu-1}(1-s^\nu)+s^\nu\big(R-1+s(r-R)\big)<0,\quad \nu\ge1.
$$
Now we are going to construct inductively a sequence of holomorphic
functions $\{h_\nu\}$. Put $h_0=1$ and $h_1=f_1$. Assume that $h_0,h_1,\dots,h_\nu$ are constructed. Put
$$
W_\nu=\{x\in X:\max_{1\le j\le \nu}|h_j(x)|\ge 1+\epsilon_\nu\}.
$$
By the assumptions there exists an $m_\nu$ such that for $h_{\nu+1}=f_{m_\nu}$ we have:
\begin{itemize}
\item $|h_{\nu+1}|\le R$ on $X$;
\item $h_{\nu+1}(x_0)=1$;
\item $|h_{\nu+1}|\le r$ on $W_\nu$.
\end{itemize}
Put $h=(1-s)\sum_{j=0}^\infty s^j h_j$. Note that $h(x_0)=1$.
Let us show that $|h|\le1$ on $X$.

Assume that $x\not\in\cup_\nu W_\nu$. Then $|h_\nu(x)|\le 1$ for any $\nu\in\NN$. Hence, $|h(x)|\le 1$.

Note that $W_1\subset W_2\subset\dots$. Now assume that
$x\in W_\nu\setminus W_{\nu-1}$ (we put $W_0=\varnothing$). Then
\begin{itemize}
\item $|h_j(x)|\le R+\epsilon_{\nu-1}$ for $1\le j\le \nu-1$;
\item $|h_\nu(x)|\le 1$;
\item $|h_j(x)|\le r$ for $j>\nu$.
\end{itemize}
So, we have
$$
|h(x)|\le\frac{1-s}{R}\Big\{ (R+\epsilon_{\nu-1})\sum_{j=0}^{\nu-1} s^j+
s^\nu+r\sum_{j=\nu+1}^\infty s^j\Big\}<1.
$$
Note that  $h$ is a non-constant function and that $h(x_0)=1$.
\end{proof}

Recall the definition of the analytic capacity. Let $\widehat\CC=\CC\cup\{\infty\}$ denote the Riemann sphere. The analytic capacity of a compact set $K$ is defined by 
$$
\gamma(K)=\sup\{|f'(\infty)|: f\in\OO(\Omega), \|f\|\le1, f(\infty)=0\},
$$
where $\Omega$ is the unbounded component of $\widehat\CC\setminus K$ and $f'(\infty)=\lim_{z\to\infty} zf(z)$. 
For any set $F\subset\CC$ we put
$$
\gamma(F)=\sup\{\gamma(K): K\subset F\text{ compacts}\}.
$$
We have the following elementary result (c.f. Corollary 1.8 in Chapter VIII \cite{Gam})
\begin{prop}\label{prop:6} If $K\subset\CC$ is a compact set then
$$
\gamma(K)=\inf\{\gamma(U): K\subset U\text{ open}\}.
$$
\end{prop}

One of the main result of this section is the following Curtis type result
(cf.~Theorem 4.1 in Chapter VIII \cite{Gam}).
\begin{theorem} Let $F\subset\CC$ be any subset and let $\zeta\in F$. Assume that
$$
\limsup_{r\to0+}\frac{\gamma(\DD(\zeta;r)\setminus F)}{r}>0.
$$
Then $\zeta$ is a peak point for $R(F)$.
\end{theorem}

\begin{proof}
There exists a sequence of holomorphic functions $f_j:\Omega_j\to\DD$ such that $f_j(\infty)=0$ and
$$
\lim_{j\to\infty} \frac{f'_j(\infty)}{r_j}>0,
$$
where $\Omega_j=\widehat\CC\setminus K_j$.
Note that $\cup_{j=1}^\infty\Omega_j=\widehat\CC\setminus\{\zeta\}$. Consider a sequence of functions
$$
g_j(z)=\frac{(z-\zeta)f_j(z)}{f'_j(\infty)},\quad z\in\Omega_j.
$$
Fix for a while $j$. From the maximum pronciple, for any $z\in\Omega_j$ with $|z-\zeta|\ge r_j$ we have $|g_j(z)|\le \frac{r_j}{f_j'(\infty)}$. Hence, 
$|g_j(z)|\le \frac2\alpha$ for any $z\in\Omega_j$ and big enough $j$. Moreover, $g_j(\infty)=1$. Passing to a subsequence, we can assume that the sequence $g_j$ converges to a bounded holomorphic function $g$ on $\widehat\CC\setminus\{\zeta\}$. Hence, $g\equiv1$. Put $h_j=1-g_j$. Then
$h_j(\zeta)=1$, $|h_j|\le 1+\frac2\alpha$, and for any $z\in\widehat\CC\setminus\{\zeta\}$ we have $h_j(z)\to0$ when $j\to\infty$.
Now we use Theorem~\ref{theorem:10a} to construct a peak function.
\end{proof}

Using the relation between the analytic capacity and the Lebesgue measure on the complex plane (see \cite{Gam}) we get.
\begin{corollary}\label{cor:6a}
Let $F\subset\CC$ be a Borel set and let $\zeta\in F$. Assume that
$$
\limsup_{r\to0+} \frac{\LL(\DD(\zeta;r)\setminus F)}{r^2}>0.
$$
Then $\zeta$ is a peak point for $R(F)$.
\end{corollary}
In another words, if $\zeta$ is not a peak point for $R(F)$ then $F$ is of full measure at $\zeta$, i.e., $\lim_{r\to0+} \frac{\LL(\DD(\zeta;r)\setminus F)}{r^2}=0$.

As a corollary of Melnikov's result  (see Theorem 4.5 in Chapter VIII in \cite{Gam})
 and Bishops's characterization of a peak point for a compact set  (see e.g. Theorem 2.1 in \cite{Zal})
we have
\begin{theorem}
Let $\Omega\subset\CC$ be any domain and let $\zeta\in\partial\Omega$. 
Assume that
$$
\sum_{n=1}^\infty\frac{\gamma(A_n(\zeta;a)\setminus\Omega)}{a^n}<+\infty.
$$
Then there exists a compact set $K\subset\Omega\cup\{\zeta\}$ such that $\zeta$ is not a peak point for $R(K)$, and, therefore,
there exists a Borel probability measure $\mu$ with support in $K$ such that
$$
f(\zeta)=\int_{\Omega} f d\mu\quad\text{ for any }f\in R(\Omega\cup\{\zeta\}).
$$
\end{theorem}

Let us finish this section with a generalization of a part of Melnikov's result.
\begin{theorem}\label{thm:8i} Let $F\subset\CC$ be any subset and let $\zeta\in F$. Let $a\in(0,1)$. If 
$$
\sum_{n=1}^\infty\frac{\gamma(A_n(\zeta,a)\setminus F)}{a^n}=+\infty
$$
then $\zeta$ is a peak point for $R(F)$.
\end{theorem}

\begin{proof} The proof essentially is the same as in case of a compact set $F$ (see sufficiency part in Melnikov's criterion, Theorem VIII.4.5 in \cite{Gam}). Indeed, from the definition it follows that for any $n\in\NN$ there exists a compact set $K_n\subset A_n(\zeta,a)\setminus F)$ and a holomorphic function $f_n:\widehat\CC\setminus K_n\to\DD$ such that $f_n(\infty)=0$ and $f_n'(\infty)\ge\frac12\gamma(A_n(\zeta,a)\setminus F)$. Having the family $\{f_n\}_{n\in\NN}$, we just repeat the arguments given on page 206 in \cite{Gam}.
\end{proof}

\section{Proof of Theorem~\ref{thm:1}}

We denote by $\MM$ the set of all positive Borel finite measures in $\CC$. For $\mu\in\MM$ we define its Newton potential as $M(z)=M_{\mu}(z)=\int\frac{1}{|w-z|}d\mu(w)$. 

For any $\zeta,\eta\in\CC$ and any $r>0$ we set
$$
F_{r}(\eta)=\frac{1}{\pi r^2}\int_{\DD(\zeta;r)}\left|\frac{z-\zeta}{z-\eta}\right|d\LL(z),
$$
where $d\LL$ is the Lebesgue measure in $\CC$. We have the estimate
$$
F_r(\eta)\le\frac{1}{\pi r}\int_{\DD(\zeta,r)}\frac{1}{|z-\eta|}d\LL(z)\le 2.
$$ 
Hence, $\int F_r(\eta)d\mu(\eta)\le 2\mu(\CC)$. So, 
\begin{equation}\label{eq:123}
\frac{1}{\pi r^2}\int_{\DD(\zeta,r)}|z-\zeta|\cdot M(z)d\LL(z)\le2\mu(\CC),
\end{equation}
and, therefore, $M<\infty$ a.e. on $\CC$. The following result, which essentially is a corollary of Fubini's theorem, shows the behaviour of the left side of \eqref{eq:123} when $r\to0$ (see e.g. \cite{St1}, Lemma 26.16).
\begin{prop}\label{prop:11} Let $\mu\in\MM$. For any $\zeta\in\CC$ we have
$$
\lim_{r\to0}\frac{1}{\pi r^2}\int_{\DD(\zeta,r)}|w-\zeta|\cdot M(w)d\LL(w)=\mu(\{\zeta\}).
$$
In particular, if $\mu(\{\zeta\})=0$ then for any $\epsilon>0$ the set
$$
\Pi(\epsilon)=\{w\in\CC: |w-\zeta|\cdot M(w)>\epsilon\}
$$
is of zero density at $\zeta$, i.e.,
$$
\lim_{r\to0}\frac{\LL(\Pi(\epsilon)\cap\DD(\zeta,r))}{\pi r^2}=0.
$$
\end{prop}

The proof of Theorem~\ref{thm:1} is based on the following simple observation.

\begin{prop}\label{thm:23a}
Let $F\subset\CC$ be any set and let $\zeta\in X$. Assume that $\mu$ is a Borel finite measure in $\CC$ with $\mu(\{\zeta\})=0$ such that
\begin{equation}\label{eq:9}
|f(\zeta)|\le\int_{F}|f|d\mu
\end{equation}
for any $f\in R(F)$. Then for any $\eta\in F$ we have
$$
|f(\eta)-f(\zeta)|\le 2\|f\|_{\infty} M(\eta)|\eta-\zeta|.
$$
\end{prop}

\begin{proof}
Fix a function $f\in R(F)$ and a point $\eta\in F\setminus\{\zeta\}$. It suffices to note that $\tilde f(z)=\frac{f(z)-f(\eta)}{z-\eta}\in R(F)$
and use the inequality \eqref{eq:9} to $\tilde f$.
\end{proof}
As an immediade corollary we get.
\begin{corollary}\label{thm:23}
Let $\Omega\subset\CC$ be a domain and let $\zeta\in\partial\Omega$. Assume that $\mu$ is a Borel finite measure in $\CC$ such that
$$
|f(\zeta)|\le\int_{\Omega}|f|d\mu
$$
for any $f\in H^\infty(\Omega)$ which extends holomorphically to a neighborhood of $\zeta$. Then for any $\eta\in\Omega$ we have
$$
|f(\eta)-f(\zeta)|\le 2\|f\|_{\infty} M(\eta)|\eta-\zeta|.
$$
In particular, for any $\eta_1,\eta_2\in\Omega$ we have
\begin{equation}\label{eq:124}
c^\ast_{\Omega}(\eta_1,\eta_2)\le 34\Big(|\zeta-\eta_1| M(\eta_1)+|\zeta-\eta_2| M(\eta_2)\Big).
\end{equation}
Moreover, there exists a c-Cauchy sequence $\{\eta_n\}_{n\ge1}\subset\Omega$ such that $\eta_n\to\zeta$.
\end{corollary}
\begin{proof} 
First part follows from Corollary~\ref{thm:23}. For the second part, if 
$$
\limsup_{r\to0}\frac{\LL(\DD(\zeta;r)\setminus\Omega)}{\pi r^2}>0
$$
then from Corollary~\ref{cor:6a} there exists a function $f\in R(\Omega\cup\{\zeta\})$ such that $|f|<1$ on $\Omega$ and $f(\zeta)=1$. In particular, the measure $\mu$ as in the assumptions does not exist. So, 
$$
\lim_{r\to0}\frac{\LL(\DD(\zeta;r)\cap\Omega)}{\pi r^2}=1.
$$
Then by Proposition~\ref{prop:11} there exists a sequence $\{\eta_n\}_{n\ge1}\subset \Omega$ with $\eta_n\to\zeta$ such that $|\zeta-\eta_n| M(\eta_n)\le\frac1{2^n}$. From Theorem~\ref{thm:23} we get the result.
\end{proof}

\begin{proof}[Proof of Theorem~\ref{thm:1}] Note that implications 
$(1)\implies(7)\implies(5)\implies(6)$ and $(2)\implies(3)$ are straightforward. The implication $(6)\implies(2)$ is proved above as a Corollary~\ref{thm:23}. 
And the implication $(4)\implies(1)$ follows from Theorem~\ref{thm:8i}.
\end{proof}

\section{A counterexample}
Let $V$ be a vector space. Following \cite{Pol-Gog} we call a mapping $Q:V\to[-\infty,+\infty)$ \emph{superlinear} if:
\begin{enumerate}
\item $Q(cv)=cQ(v)$ for any $c\ge0$ and any $v\in V$;
\item $Q(v_1+v_2)\ge Q(v_1)+Q(v_2)$.
\end{enumerate}
In \cite{Pol-Gog} the following version of a Hahn-Banach theorem is stated:

{\it Let $V$ be a vector space and let $M\subset V$ be a vector subspace. Assume that $Q:V\to[-\infty,+\infty)$ be a superlinear mapping. Then any linear operator $\ell:M\to\RR$ such that $\ell\ge Q$ on $M$ extends to a linear operator $L:V\to\RR$ such that $L\ge Q$ on $V$.}

The following elementary counterexample shows that it does not hold.
\begin{example}
Let $V=\RR^2$ and let $M=\RR\times\{0\}$. We put
$$
Q(x,y)=
\begin{cases}
-\infty\quad&\text{ if }y\le0\\
0&\text{ if }y>0
\end{cases}
$$
and $\ell(x,y)=x$. We have $\ell\ge Q$ on $M$. Moreover, $Q$ is superlinear. However, there is no linear $L:V\to\RR$ 
such that $L=\ell$ on $M$ and $L\ge Q$ on $V$.
\end{example}

\bibliographystyle{amsplain}

\begin{thebibliography}{10}


\bibitem{Edi} A.~Edigarian, \textit{Peak points for domains in $\CC^n$}, Ann. Polon. Math., 114.1 (2005), 1--12.

\bibitem{Gam} T.W.~Gamelin, \textit{Uniform algebras}, Chelsea Publishing Company, 1984.

\bibitem{Gam-Gar} T.W.~Gamelin and J.~Garnett, \textit{Distinguished homomorphisms and fiber algebras}, Amer. J. Math 92 (1970), 455--474.

\bibitem{Pol-Gog} N.G.~Gogus, T.L.~Perkins and E.A.~Poletsky, \textit{Non-compact versions of Edwards' theorem}, Positivity 17 (2013), 459--473.

\bibitem{JP1} M.~Jarnicki \& P.~Pflug, \textit{Invariant distances and metrics in complex analysis}, De Gruyter, 2nd Extended Edition, 2013.



\bibitem {Sel1} M.A.~Selby, \textit{On completeness with respect to the Carath\'eodory metric}, Canad. Math. Bull. 17 (1974), 261--263.

\bibitem {Sib1} N.~Sibony, \textit{Prolongement de fonctions holomorphes born\'ees et metrique de Carath\'eodory}, Invent. Math. 29 (1975), 205--230

\bibitem {St1} E.L.~Stout, \textit{The theory of uniform algebras}, Bogden and Quigley Publishers, 1971.

\bibitem{Zal} L.~Zalcman, \textit{Analytic capacity and rational approximation}, Lecture Notes in Math. 50, Springer Verlag, Berlin, 1968.


\end{thebibliography}


\end{document}